\RequirePackage{fix-cm}
\documentclass[smallextended]{svjour3}       % onecolumn (second format)
\smartqed  % flush right qed marks, e.g. at end of proof
\usepackage{graphicx}
\usepackage{amsmath}
\usepackage{amssymb}
\usepackage{amsfonts}
\usepackage{mathrsfs}

\newcommand{\R}{\mathbb{R}}

\newcommand{\A}{\mathcal{A}}

\newtheorem{thm}{Theorem}[section]
\newtheorem{prop}[thm]{Proposition}
\newtheorem{lem}[thm]{Lemma}
\newtheorem{rem}[thm]{Remark}
\newtheorem{cor}[thm]{Corollary}

\newtheorem{defn}[thm]{Definition}

\begin{document}

\title{A variant of H\"ormander's $L^2$ existence theorem for Dirac operator in Clifford analysis%\thanks{Grants or other notes
}


\author{Yang Liu          \and
        Zhihua Chen \and Yifei Pan %etc.
}


\institute{Yang Liu \at
              Department of Mathematics, Zhejiang Normal
                University, Jinhua 321004, China \\
              Fax: +86-57982298897\\
              \email{liuyang@zjnu.edu.cn; liuyang4740@gmail.com}           %  \\
           \and
           Zhihua Chen \at
              Department of Mathematics, Tongji University, Shanghai 200092, China\\
              \email{zzzhhc@tongji.edu.cn}
           \and
           Yifei Pan \at
              Department of Mathematical Sciences, Indiana University-Purdue University Fort Wayne, Fort Wayne, Indiana 46805, USA\\
              \email{pan@ipfw.edu}
}

\date{Received: date / Accepted: date}

\maketitle

\begin{abstract}
In this paper, we give the H\"ormander's $L^2$ theorem for Dirac operator over an open subset $\Omega\in\R^{n+1}$ with Clifford algebra. Some sufficient condition on the existence of the weak solutions for Dirac operator has been found in the sense of Clifford analysis. In particular, if $\Omega$ is bounded, then we prove that for any $f$ in $L^2$ space with value in Clifford algebra, there exists a weak solution of Dirac operator such that $$\overline{D}u=f$$ with $u$ in the $L^2$ space as well. The method is based on H\"ormander's $L^2$ existence theorem in complex analysis and the $L^2$ weighted space is utilised.
\end{abstract}

\section{Introduction}
The  development  of  function  theories  on  Clifford  algebras  has  proved  a
useful  setting  for  generalizing  many
aspects  of  one  variable  complex  function  theory to  higher  dimensions.
The  study  of  these  function  theories  is  referred  to  as  Clifford  analysis  \cite{c,c2,qt,J5}, which  is  closely  related  to  a  number  of  studies  made  in
mathematical  physics,  and  many  applications  in  this  area
have  been  found  in  recent  years. In \cite{J4}, Ryan considered solutions of the polynomial Dirac operator, which afforded an integral representation. Furthermore, the author gave a Pompeiu representation for $C^1$-functions in a Lipschitz bounded domain. In \cite{J2}, the author presented a classification of linear, conformally invariant, Clifford-algebra-valued differential operators over $\mathbb{C}^n$, which comprised the Dirac operator and its iterates. In \cite{J6}, Qian and Ryan used Vahlen matrices to study the conformal covariance of various types of Hardy spaces over hypersurfaces in $\mathbb{R}^n$. In \cite{mz}, the discrete Fueter polynomials was introduced, which formed a basis of the space of discrete spherical monogenics. Moreover, the explicit construction for this discrete Fueter basis, in arbitrary dimension $m$ and for arbitrary homogeneity degree $k$ was presented as well.

In \cite{H}, the famous H\"ormander's $L^2$ existence and approximation
theorems was given for the $\bar{\partial}$ operator in pseudo-convex domains in $\mathbb{C}^n$. When $n=1$, the existence theorem of complex variable can be deduced. The aim of this paper is to establish a H\"ormander's $L^2$ theorem in $\R^{n+1}$ with Clifford analysis, and present sufficient condition on the existence of the weak solutions for Dirac operator in the sense of Clifford algebra.

Let $\A$ be a real Clifford algebra over an (n+1)-dimensional real vector space $\R^{n+1}$ and the corresponding norm on $\A$ is given by $|\lambda|_0=\sqrt{(\lambda,\lambda)_0}$ (see subsection \ref{sub1}). Let $\Omega$ be an open subset of $\R^{n+1}$, $L^2(\Omega,\A,\varphi)$ be a right Hilbert $\A$-module for a given function $\varphi\in C^2(\Omega, \R)$ with the norm given by Definition \ref{defn10}. (see subsection \ref{sub3}).  $\overline{D}$ denotes the Dirac differential operator and the dual operator  $\overline{D}^*_\varphi $ of $\overline{D}$  is given by (\ref{7}). For $x=(x_0,x_1,...,x_n)\in \R^{n+1}$, $\Delta=\sum_{i=0}^{n}\frac{\partial^2}{\partial x_i^2}$. Then we can obtain our main results as follows.

\begin{thm}\label{thm1}
Given $f\in L^2(\Omega,\A,\varphi)$, there exists $u\in L^2(\Omega,\A,\varphi)$ such that
\begin{equation}\label{eq:5.1}
\begin{split}\overline{D}u=f\end{split}
\end{equation} with
\begin{equation}\label{eq:5.2}
\begin{split}\|u\|^2=\int_\Omega|u|^2_0e^{-\varphi}dx\leq 2^{2n}c \end{split}
\end{equation} if
\begin{equation}\label{eq:5}
\begin{split}|(f,\alpha)_\varphi|^2_0\leq c\|\overline{D}^*_\varphi\alpha\|^2=c\int_\Omega|\overline{D}^*_\varphi\alpha|^2_0e^{-\varphi}dx,~\forall \alpha\in C^\infty_0(\Omega,\A).\end{split}
\end{equation}
Conversely, if there exists $u\in L^2(\Omega,\A,\varphi)$ such that (\ref{eq:5.1}) is satisfied with
\begin{equation}
\begin{split}\|u\|^2 \leq  c \nonumber\end{split}
\end{equation}
Then we can get the inequality (\ref{eq:5}) for norm estimation.
\end{thm}

The factor $2^{2n}$ in (\ref{eq:5.2}) comes from the definition of the norm in Clifford analysis. If $n=1$, then the factor would disappear which gives a necessary and sufficient condition in the theorem.
From the above theorem, we give the following sufficient condition on the existence of weak solutions for Dirac operator.
\begin{thm}\label{thm2}
Given $\varphi\in C^2(\Omega,\mathbb{R})$ and $n> 1$; $\Delta\varphi\geq0$,~and~$\frac{\partial^2 \varphi}{\partial x_j\partial x_i}=0,~i\neq j,~1\leq i,j\leq n$ and $\frac{\partial^2 \varphi}{\partial x^2_i}\leq 0,~1\leq i\leq n$. Then for all $ f\in L^2(\Omega,\A,\varphi)$ with $\int_\Omega\frac{|f|^2_0}{\Delta\varphi}e^{-\varphi}dx=c<\infty$, there exists a $u\in L^2(\Omega,\A,\varphi)$ such that $$\overline{D}u=f$$ with
$$\|u\|^2=\int_\Omega|u|^2_0e^{-\varphi}dx\leq2^{2n}\int_\Omega\frac{|f|^2_0}{\Delta\varphi}e^{-\varphi}dx.$$
\end{thm}

\begin{rem}
 Assuming $x=(x_0,x_1,...,x_n)\in \R^{n+1}$, it is easy to see that $\varphi(x)=x_0^2$ satisfies the conditions in Theorem \ref{thm2}. Another simple example would be
$$\varphi(x)=(n+1)x_0^2-\sum_{i=1}^{n}x_i^2.$$
It is obvious that $\Delta\varphi(x)=2$, $\frac{\partial^2 \varphi}{\partial x^2_i}=-2$, and $\frac{\partial^2 \varphi}{\partial x_j\partial x_i}=0,~i\neq j,~1\leq i,j\leq n$.

\end{rem}
\begin{cor}\label{cor1}
Given $\varphi\in C^2(\Omega,\mathbb{R}),$ and $\varphi(x)=\varphi(x_0)$ with $\varphi''(x_0)\geq0$. Then for all $ f\in L^2(\Omega,\A,\varphi)$ with $\int_\Omega\frac{|f|^2_0}{\varphi''}e^{-\varphi}dx=c<\infty$, there exists a $u\in L^2(\Omega,\A,\varphi)$ such that $$\overline{D}u=f$$ with
$$\|u\|^2=\int_\Omega|u|^2_0e^{-\varphi}dx\leq2^{2n}\int_\Omega\frac{|f|^2_0}{\varphi''}e^{-\varphi}dx.$$
\end{cor}

It is noticed that there is nothing to do with the boundary conditions of $\Omega$ in the above results. This phenomenon is totally different with the famous  H\"ormander's $L^2$ existence theorems of several complex variables in \cite{H}. Then we can also have the following theorem on global solutions.

\begin{thm}\label{cor3}
Given $\varphi\in C^2(\R^{n+1},\mathbb{R})$ with all derivative conditions in Theorem \ref{thm1} satisfied. Then for all $ f\in L^2(\R^{n+1},\A,\varphi)$ with $\int_{\R^{n+1}}\frac{|f|^2_0}{\Delta\varphi}e^{-\varphi}dx=c<\infty$, there exists a $u\in L^2(\R^{n+1},\A,\varphi)$  satisfying $$\overline{D}u=f$$ with
$$\|u\|^2=\int_{\R^{n+1}}|u|^2_0e^{-\varphi}dx\leq2^{2n} \int_{\R^{n+1}}\frac{|f|^2_0}{\Delta\varphi}e^{-\varphi}dx.$$
\end{thm}

On the other hand, if the boundary of $\Omega$ is concerned, we consider a special kind of domain ${\Omega}_0=\{x\in\R^{n+1}:a\leq x_0\leq b\}$ for any $a,~b\in \R$ with $a<b$, then we can get the following theorem within $L^2$ space instead of $L^2$ weighted space.
\begin{thm}\label{thm5}
Let $ f\in L^2({\Omega_0},\A)$. Then there exists a $u\in L^2({\Omega_0},\A)$ such that $$\overline{D}u=f$$ with
$$\int_{\Omega_0}|u|^2_0dx\leq2^{2n}c(a,b)\int_{\Omega_0}{|f|^2_0}dx$$ and $c(a,b)$ is a factor depending on $a,~b$.
\end{thm}
\begin{proof}
Let $\varphi(x)=x_0^2$. It can be obtained that $L^2({\Omega_0},\A)=L^2({\Omega_0},\A,\varphi)$ for the boundary of $x_0$.
Then the theorem is proved with Theorem \ref{thm2}.
\end{proof}

\begin{rem}
In particular, any bounded domain $\Omega$ in $\R^{n+1}$ can be regarded as one type of $\Omega_0$. Therefore, it comes from Theorem \ref{thm5} that for any $f\in L^2(\Omega,\A)$, we can find a weak solution of Dirac operator $\overline{D}u=f$ with $u\in L^2(\Omega,\A)$.
\end{rem}

\section{Preliminaries}
To make the paper self-contained, some basic notations and results used in this paper are included.
\subsection{The Clifford algebra $\A$ }\label{sub1}

Let $\A$ be a real Clifford algebra over an (n+1)-dimensional real vector space $\R^{n+1}$ with orthogonal basis $e:=\{e_0,e_1,...,e_n\}$, where $e_0=1$ is a unit element in $\R^{n+1}$. Furthermore,
 \begin{equation} \label{eq:1}
 \left\{ \begin{aligned}
e_ie_j+e_je_i&= 0,~i\neq j \\
 e_i^2&= -1,~i=1,...,n.
 \end{aligned} \right.\nonumber
 \end{equation}
Then $\A$ has its basis $$\{e_A=e_{h_1\cdots h_r}=e_{h_1}\cdots e_{h_r}: 1\leq h_1<...<h_r\leq n, 1\leq r\leq n\}.$$
If $i\in \{h_1,...,h_r\}$, we denote $i\in A$ and if $i\not\in \{h_1,...,h_r\}$, we denote $i\not\in A$. $A-{i}$ means $\{h_1,...,h_r\}\setminus\{i\}$ and $A+{i}$ means $\{h_1,...,h_r\}\cup\{i\}$.
So the real Clifford algebra is composed of elements having the type $a=\sum\limits_{A}x_Ae_A$, in which $x_A\in \R$ are real numbers.
For $a\in \A$, we give the inversion in the Clifford algebra as follows:
$a^*=\sum\limits_{A}x_Ae_A^*$
where $e_A^*=(-1)^{|A|}e_A$ and $|A|=n(A)$ is the $r\in\mathbb{Z}^+$ as $e_A=e_{h_1\cdots h_r}$. When $A=\emptyset$, $e_A=e_0$, $|A|=0$. Next, we define the reversion in the Clifford algebra, which is given by
$a^\dag=\sum\limits_{A}x_Ae_A^\dag$
where $e_A^\dag=(-1)^{(|A|-1)|A|/2}e_A.$ Now we present the involution which is a combination of the inversion and the reversion introduced above.
$$\bar{a}=\sum\limits_{A}x_A\bar{e}_A$$
where $\bar{e}_A=e_A^{*\dag}=(-1)^{(|A|+1)|A|/2}e_A.$ From the definition, one can easily deduce that $e_A\bar{e}_A=\bar{e}_Ae_A=1.$ Furthermore, we have $$\overline{\lambda\mu}=\bar{\mu}\bar{\lambda},~~\forall \lambda, \mu\in \A.$$
Let $a=\sum\limits_{A}x_Ae_A$ be a Clifford number. The coefficient $x_A$ of the $e_A$-component will also be denoted by $[a]_A$. In particular the coefficient $x_0$ of the $e_0$-component will be denoted by $[a]_0$, which is called the scalar part of the Clifford number $a$. An inner product on $\A$ is defined by putting for any $\lambda,\mu\in \A$, $(\lambda,\mu)_0:=2^n[\lambda\bar{\mu}]_0=2^n\sum\limits_{A}\lambda_A\mu_A$. The corresponding norm on $\A$ reads $|\lambda|_0=\sqrt{(\lambda,\lambda)_0}$.

We define a real functional on $\A$ that $\tau_{e_A}:\A \rightarrow \R$
$$\langle \tau_{e_A},\mu\rangle=2^n(-1)^{(|A|+1)|A|/2}\mu_A.$$ In the special case where $A=\emptyset$ we have
$$\langle \tau_{e_0},\mu\rangle=2^n\mu_0.$$

Let $\Omega$ be an open subset of $\R^{n+1}$. Then functions $f$ defined in $\Omega$ and with values in $\A$ are considered. They are of the form
$$f(x)=\sum_{A}f_A(x)e_A$$ where $f_A(x)$ are functions with real value. Let $\overline{D}$ denotes the Dirac differential operator
$$\overline{D}=\sum_{i=0}^{n}e_i\partial_{x_i},$$
its action on functions from the left and from the right being governed by the rules
$$\overline{D}f=\sum_{i,A}e_ie_A\partial_{x_i}f_A~\mbox{and}~f\overline{D}=\sum_{i,A}e_Ae_i\partial_{x_i}f_A.$$
$f$ is called left-monogenic if $\overline{D}f=0$ and it is called right-monogenic if $f\overline{D}=0$. The conjugate operator is given by
$$D=\sum_{i=0}^{n}\bar{e}_i\partial_{x_i}.$$ It can be found that $$\overline{D}D=D\overline{D}=\Delta$$ where $\Delta$ denotes the classical Laplacian in $\R^{n+1}$.
When $n=1$, one can think of $x_0$ as the real part and of $x_1$ as the imaginary part of the variable and to identify $e_1$ with $i$. the operator $\overline{D}$ then take the form $\overline{D}=\partial_{x_0}+i\partial_{x_1}$, which is similar with the operator $\bar{\partial}$ in complex analysis.

\subsection{Modules over Clifford algebras} This subsection is to give some general information concerning a class of topological modules over Clifford algebras. In the sequel definitions and properties will be stated for left $\A$-module and their duals, the passage to the case of right $\A$-module being straight-forward.

\begin{defn}{\bf (unitary left $\A$-module)}
Let $X$ be a unitary left $\A$-module, i.e. $X$ is abelian group and a law $(\lambda,f)\rightarrow\lambda f:\A\times X\rightarrow X$ is defined such that $\forall\lambda,\mu\in \A$, and $f,~g\in X$
\begin{enumerate}
  \item [(1)] $(\lambda+\mu)f=\lambda f+\mu f$,
  \item [(2)] $\lambda\mu f=\lambda(\mu f)$,
  \item [(3)] $\lambda(f+g)=\lambda f+\lambda g$,
  \item [(4)] $e_0 f=f$.
\end{enumerate}
Moreover, when speaking of a submodule $E$ of the unitary left $\A$-module $X$, we mean that $E$ is a non empty subset of $X$ which becomes a unitary left $\A$-module too when restricting the module operations of $X$ to $E$.
\end{defn}

\begin{defn}{\bf (left $\A$-linear operator)}
If $X,Y$ are unitary left $\A$-modules, then $T:X\rightarrow Y$ is said to be a left $\A$-linear operator,
if $\forall~f,~g\in X$ and $\lambda\in \A$
$$T(\lambda f+g)=\lambda T(f)+T(g).$$
The set of all $``T"$ is denoted by $L(X,Y)$. If $Y=\A,~L(X,\A)$ is called the algebraic dual of $X$ and denoted by $X^{*alg}$. Its elements are called left $\A$-linear functionals on $X$ and for any $T\in X^{*alg}$ and $f\in X$, we denote by $\langle T,f \rangle$ the value of $T$ at $f$.
\end{defn}

\begin{defn}\label{bounded}{\bf (bounded functional)}
An element $T\in X^{*alg}$ is called bounded, if there exist a semi-norm $p$ on $X$ and $c>0$ such that for all $f\in X$
$$|\langle T,f\rangle|_0\leq c\cdot p(f).$$
\end{defn}

\begin{thm}\label{Hahn}{\bf (Hahn-Banach type theorem)}\cite{c}
Let $X$ be a unitary left $\A$-module with semi-norm $p$, $Y$ be a submodule of $X$, and $T$ be a left $\A$-linear
functional on $Y$ such that for some $c>0,$ $$|\langle T,g\rangle|_0\leq c\cdot p(g),~~ \forall g\in Y$$
Then there exists a left $\A$-linear functional $\widetilde{T}$ on $X$ such that
\begin{enumerate}
  \item [(1)] $\widetilde{T}\mid_Y=T$,
  \item [(2)] for some $c^*>0$,~$|\langle \widetilde{T},f\rangle|_0\leq c^*\cdot p(f)$,~~$\forall f\in X$.
\end{enumerate}
\end{thm}

\begin{defn}{\bf (inner product on a unitary right $\A$-module)}
Let $H$ be a unitary right $\A$-module, then a function $(~,~):~H\times H\rightarrow \A$ is said to be a inner product on $H$ if for all $ f,g,h\in H$ and $\lambda\in \A$,
\begin{enumerate}
  \item [(1)] $(f,g+h)=(f,g)+(f,h)$,
  \item [(2)] $(f,g\lambda)=(f,g)\lambda$,
  \item [(3)] $(f,g)=\overline{(g,f)}$,
  \item [(4)] $\langle\tau_{e_0},(f,f)\rangle\geq0$ and $\langle\tau_{e_0},(f,f)\rangle=0~ \mbox{if and only if} ~f=0$,
  \item [(5)] $\langle\tau_{e_0},(f\lambda,f\lambda)\rangle\leq|\lambda|^2_0\langle\tau_{e_0},(f,f)\rangle$.
\end{enumerate}\end{defn}
From the definition on inner product, putting for each $f\in H$
$$\|f\|^2=\langle\tau_{e_0},(f,f)\rangle,$$
then it can be obtained that for any $f,g\in H,$
\begin{equation} \label{eq:2}
\begin{split}
|\langle\tau_{e_0},~(f,g)\rangle|\leq\|f\|\|g\|,\|f+g\|\leq\|f\|+\|g\|.
 \end{split}\nonumber
 \end{equation}
 Hence, $\|\cdot\|$ is a proper norm on $H$ turning it into a normed right $A$-module.
Moreover, we have the following Cauchy-Schwarz inequality.
\begin{prop}\cite{c}\label{prop1}
For all $ f,g\in H,$ $|(f,g)|_0\leq\|f\|\|g\|.$
\end{prop}

\begin{defn}{\bf (right Hilbert $\A$-module)}
Let $H$ be a unitary right $\A$-module provided with an inner product $(~,~)$. Then is it called a right Hilbert $\A$-module if it is complete for the norm topology derived from the inner product.
\end{defn}

\begin{thm}\label{Riesz}{\bf (Riesz representation theorem)}\cite{c}
Let $H$ be a right Hilbert $\A$-modules and $T\in H^{*alg}$. Then $T$ is bounded if and only if there exists a (unique) element $g\in H$ such that for all $f\in H$,
$$T(f):=\langle T,f\rangle=(g,f).$$
\end{thm}

\subsection{Hilbert space of square integrable functions}\label{sub3}
Now we extend the standard Hilbert space of square integrable functions to Clifford algebra. First, we denote $L^1(\Omega,\mu)$ and $L^2(\Omega,\mu)$ be the sets of all integrable or square integrable functions defined on the domain $\Omega\in \R^{n+1}$ with respect to the measure $\mu$. Then $L^1(\Omega,\A,\mu)$ and $L^2(\Omega,\A,\mu)$ are defined as the sets of functions $f:\Omega\rightarrow \A$ which are integrable or square integrable with respect to $\mu$, i.e., if $f=\sum\limits_Af_Ae_A$, then for each $A$, $f_A\in L^1(\Omega,\mu)$ and $f^2_A\in L^1(\Omega,\mu)$, respectively. Then {\bf one may easily check that $L^1(\Omega,\A,\mu)$ and $L^2(\Omega,\A,\mu)$ are unitary bi-$\A$-module, i.e., unitary left-$\A$-module and unitary right-$\A$-module}. Furthermore, for any $f,g\in L^2(\Omega,\A,\mu)$, $\bar{f}\in L^2(\Omega,\A,\mu)$ while $\bar{f}g\in L^1(\Omega,\A,\mu)$, where $\bar{f}(x)=\overline{f(x)}$ and $(\bar{f}g)(x)=\bar{f}(x)g(x),~x\in \Omega$. Consider as a right $\A$-module, define for $f,g\in L^2(\Omega,\A,\mu)$ that
$$(f,g)=\int_{\Omega}\bar{f}(x)g(x)d\mu.$$
Furthermore for any real linear functional $T$ on $\A$
$$\langle T,(f,g)\rangle=\langle T,\int_{\Omega}\bar{f}(x)g(x)d\mu\rangle=\int_{\Omega}\langle T,\bar{f}(x)g(x)\rangle d\mu.$$
Consequently, taking $T=\tau_{e_0}$ we find that
\begin{equation} \label{eq:3}
\begin{split}
\langle \tau_{e_0},(f,f)\rangle&=\langle \tau_{e_0},\int_{\Omega}\bar{f}(x)f(x)d\mu\rangle=\int_{\Omega}\langle \tau_{e_0},\bar{f}(x)f(x)\rangle d\mu\\&=\int_{\Omega}|f(x)|^2_0d\mu.
 \end{split}\nonumber
 \end{equation}
Hence, for all $f\in L^2(\Omega,\A,\mu)$, $\langle \tau_{e_0},(f,f)\rangle\geq 0$ and $\langle \tau_{e_0},(f,f)\rangle=0$ if and only if $f=0$ a.e. in $\Omega$. Then it is easy to see that under the inner product defined, all conditions for $L^2(\Omega,\A,\mu)$ to be a unitary right inner product $\A$-module are satisfied. Since $L^2(\Omega,\A,\mu)=\prod_{A}L^2(\Omega,\mu)$, we have that $L^2(\Omega,\A,\mu)$ is complete; in other words $L^2(\Omega,\A,\mu)$ is a right Hilbert $\A$-module, with the norm $$\|f\|^2=\langle \tau_{e_0},(f,f)\rangle=\int_{\Omega}|f(x)|^2_0d\mu$$ for $f\in L^2(\Omega,\A,\mu)$.

\begin{defn}\label{defn10}{\bf (weighted $L^2$ space)}
Similar with $L^2(\Omega,\A,\mu)$, we can define the weighted $L^2(H,\A,\varphi)$ for a given function $\varphi\in C^2(\Omega, \R)$. First, let
$$L^2(\Omega,\varphi)=\big\{f|f:\Omega\rightarrow \R,~\int_\Omega|f(x)|^2e^{-\varphi}~dx<+\infty\big\}.$$Then we denote
$$L^2(H,\A,\varphi)=\{f|f:\Omega\rightarrow \A,~f=\sum\limits_Af_Ae_A,~f_A\in L^2(\Omega,\varphi)\}.$$
Moreover, for all $f,g\in L^2(H,\A,\varphi)$, we define
$$(f,g)_\varphi=\int_\Omega \bar{f}(x)g(x)e^{-\varphi}dx.$$
Then it is also easy to see $L^2(\Omega,\A,\varphi)$ is a right Hilbert $\A$-module, with the norm
\begin{equation} \label{norm}
\begin{split}
\|f\|^2=\langle \tau_{e_0},(f,f)_\varphi\rangle=\int_{\Omega}|f(x)|^2_0e^{-\varphi}dx
\end{split}\nonumber
 \end{equation}
 for $f\in L^2(\Omega,\A,\varphi)$.
\end{defn}

\subsection{Cauchy's integral formula}
Let $M$ be an (n+1)-dimensional differentiable and oriented manifold contained in some open subset $\Sigma$ of $\R^{n+1}$. By means of the n-forms
$$d\hat{x}_i=dx_0\wedge\cdots\wedge dx_{i-1}\wedge dx_{x_{i+1}}\wedge \cdots \wedge dx_n,~i=0,1,...,n,$$
an $\A$-valued n-form is introduced by putting
$$d\sigma=\sum_{i=0}^{n}(-1)^ie_id\hat{x}_i,$$ similarly, denote
$$d\bar{\sigma}=\sum_{i=0}^{n}(-1)^i\bar{e}_id\hat{x}_i.$$
%If $dS$ stands for the classical surface element and $n=\sum_{i=0}^{n}e_in_i$ where $n_i$ is the i-th component of the outward pointing normal, then the $\A$-valued surface element $d\sigma$ can be written as $d\sigma=n\cdot dS$.
Furthermore the volume-element $$dx=dx_0\wedge\cdots\wedge dx_n$$ is used.
\begin{prop}{\bf (Stokes-Green Theorem)}\cite{c} If $f,g\in C^1(\Sigma,\A)$ then for any (n+1)-chain $\Omega$ on $M\subset \Sigma$,
$$\int_{\partial\Omega}fd\sigma g=\int_\Omega(f\overline{D})gdx+\int_\Omega f(\overline{D}g)dx,$$
$$\int_{\partial\Omega}fd\bar{\sigma} g=\int_\Omega(fD)gdx+\int_\Omega f(Dg)dx.$$
\end{prop}
\begin{rem}
Denote $C^\infty_0(\Omega,\R)$ as the set of all smooth real-valued functions with compact support in $\Omega$ and $C^\infty_0(\Omega,\A):=\{f|f:\Omega\rightarrow \A,~f=\sum\limits_Af_Ae_A,~f_A\in C^\infty_0(\Omega,\R)\}.$ If $f$ or $g\in C^\infty_0(\Omega,\A)$, then we have from the Stokes-Green theorem that
$$\int_\Omega(f\overline{D})gdx=-\int_\Omega f(\overline{D}g)dx,$$
$$\int_\Omega(fD)gdx=-\int_\Omega f(Dg)dx.$$
\end{rem}
\begin{lem}
If $u(x)\in C^1(\Omega,\A)$, then $\overline{\overline{D}u}=\bar{u}D$.
\end{lem}
\begin{proof}
Let $u(x)=\sum_{A}e_Au_A$. Then
\begin{equation} \label{eq:4}
\begin{split}
\overline{\overline{D}u}=\sum_{i,A}\overline{e_ie_A}\partial_{x_i}u_A
=\sum_{i,A}\bar{e}_A\bar{e}_i\partial_{x_i}u_A
=\bar{u}D.
 \end{split}\nonumber
 \end{equation}
\end{proof}

\begin{lem}\cite{c2}
If $u(x)=\sum_{A}e_Au_A$, $v(x)=\sum_{i=0}^{n}e_iv_i$, then
$$\overline{D}(uv)=(\overline{D}u)v+u(\overline{D}v)+\sum\limits^n_{j=1}(e_ju-ue_j)\partial_{x_j} v.$$
\end{lem}

\subsection{Weak solutions}

Let $L_{loc}^1(\Omega,\A):=\{f|f:\Omega\rightarrow \A, ~f=\sum\limits_Af_Ae_A,~f_A\in L_{loc}^1(\Omega,\R)\}$. Then we define the weak solution in the sense of Clifford algebra as follows.
\begin{defn}\label{defn1}{\bf ($\overline{D}$ solution in weak sense)}
If $f\in L_{loc}^1(\Omega,\A)$, $u:\Omega \rightarrow \A$ is a weak solution of
$$\overline{D}u=f ~(\mbox{or}~{D}u=f)$$
if for any $\alpha\in C^\infty_0(\Omega,\A)$,
$$\int_{\Omega}\alpha f dx=-\int_{\Omega}(\alpha\overline{D})udx~(\mbox{or}~\int_{\Omega}\alpha f dx=-\int_{\Omega}(\alpha {D})udx).$$
\end{defn}
It should be noticed that if $u$ is a weak solution of Dirac equation $\overline{D}u=0$, in addition, if $u$ is smooth in $\Omega$, then it is left-monogenic.
Now it is natural to give the definition of $\Delta$ solution in the weak sense.
\begin{defn}\label{defn1.1}{\bf ($\Delta$  solution in weak sense)}
If $f\in L_{loc}^1(\Omega,\A)$, $u:\Omega \rightarrow \A$ is a weak solution of
$$\Delta u=f$$
if for any $\alpha\in C^\infty_0(\Omega,\A)$,
$$\int_{\Omega}\alpha f dx=\int_{\Omega}({\Delta}\alpha)udx.$$
\end{defn}

\begin{thm}\label{thm4}
If $f\in L_{loc}^1(\Omega,\A)$, and  $\overline{D}f=0$ in weak sense, then $f$ is left-monogenic  at any point of $\Omega$.
\end{thm}
\begin{proof}: Since $\overline{D}f=0$ in weak sense, then $\Delta f=0$ in weak sense. By Weyl's lemma, $f$ is smooth in $\Omega$ and has
$\Delta f=0$ in classical sense, then of course $f$ is left-monogenic at any point of $\Omega$.
\end{proof}
\begin{rem}\label{rem1}
This is useful to deal with uniqueness of weak solutions.
for example, if $ u,~ v\in L_{loc}^1(\Omega,\A)$ are two weak solutions of $\overline{D }u=f$, then $ u=v+w$ with any $w$ left-monogenic.
\end{rem}

\begin{rem}
An important example of a left monogenic function  is the generalized Cauchy kernel $$G(x)=\frac{1}{\omega_{n+1}}\frac{\overline{x}}{|x|^{n+1}},$$
where $\omega_{n+1}$ denotes the surface area of the unit ball in $\R^{n+1}$. This function obviously belongs to $L_{loc}^1(\Omega,\A)$ and is a fundamental solution of the Dirac equation in the classical sense at any point of $\R^{n+1}$ except 0. However, it is not a weak solution of the Dirac operator. In fact, if it satisfies $\overline{D}f=0$ in the weak sense, then from Theorem \ref{thm4}, it must
be left-monogenic in the any point of $\Omega$ which could include $0$. Therefore, we get a contradiction.
\end{rem}

%if u, v are two weak solutions of $\bar$ D u=f in weak sense, then u=v+cliford analytic

For $f\in L^2(\Omega,\A,\varphi)$, $u:\Omega \rightarrow \A$. If $\overline{D}u=f$, based on the Stokes-Green theorem, we can define the dual operator $\overline{D}^*_\varphi$ of $\overline{D}$ under the inner product of $L^2(\Omega,\A,\varphi)$. For any  $\alpha\in C^\infty_0(\Omega,\A)$,
\begin{equation}\label{7}
\begin{split}
(\alpha,f)_\varphi=&~\int_\Omega\bar{\alpha}fe^{-\varphi}dx=\int_\Omega\bar{\alpha}e^{-\varphi}fdx\\
=&~\int_\Omega(\bar{\alpha}e^{-\varphi})(\overline{D}u)dx\\
=&~-\int_\Omega\big((\bar{\alpha}e^{-\varphi})\overline{D}\big)udx\\
=&~-\int_\Omega\big((\bar{\alpha}e^{-\varphi})\overline{D}\big)e^\varphi ue^{-\varphi}dx\\
=&~\int_\Omega\overline{-e^{\varphi}D(\alpha e^{-\varphi})}ue^{-\varphi}dx\\
=&~(-e^{-\varphi}D(\alpha e^{-\varphi}),u)_\varphi\triangleq(\overline{D}^*_\varphi\alpha,u)_\varphi,
\end{split}
\end{equation}
where $\overline{D}^*_\varphi\alpha=-e^\varphi D(\alpha e^{-\varphi})=\alpha (D\varphi)-D\alpha$,
i.e. $$(\alpha,\overline{D}u)_\varphi=(\overline{D}^*_\varphi\alpha,u)_\varphi.$$ In the same way, we also have
$$(\overline{D}u,\alpha)_\varphi=(u,\overline{D}^*_\varphi\alpha)_\varphi.$$

\section{The proof of Theorem \ref{thm1}}
Now we are in the position of proving Theorem \ref{thm1}.
\begin{proof}($Sufficiency$) From the definition of dual operator and Cauchy-Schwarz inequality in Proposition \ref{prop1}, we have
\begin{equation}
\begin{split}
|(f,\alpha)_\varphi|^2_0
=&|(\overline{D}u,\alpha)_\varphi|^2_0
=|(u,\overline{D}^*_\varphi\alpha)_\varphi|^2_0\\
\leq&~\|u\|^2\cdot\|\overline{D}^*_\varphi\alpha\|^2\\
\leq&~c\cdot\|\overline{D}^*_\varphi\alpha\|^2.\nonumber
\end{split}
\end{equation}
~\\
($necessity$) We aim to prove the necessity with Riesz representation theorem. First, we denote the submodule
$$E=\{\overline{D}^*_\varphi\alpha,~\alpha\in C^\infty_0(\Omega,\A),~\varphi\in C^2(\Omega,\R)\}\subset L^2(\Omega,\A,\varphi).$$
Then we define a linear functional $L_f$ on $E$, i.e., $L_f\in E^{*alg}$ for a fixed $f\in  L^2(\Omega,\A,\varphi)$ as follows,
$$\langle L_f,\overline{D}^*_\varphi\alpha\rangle=(f,\alpha)_\varphi=\int_\Omega\bar{f}\cdot\alpha\cdot e^{-\varphi}dx\in \A.$$
From (\ref{eq:5}), we have
$$|\langle L_f,\overline{D}^*_\varphi\alpha\rangle|_0=|(f,\alpha)_\varphi|_0\leq\sqrt{c}\cdot\|\overline{D}^*_\varphi\alpha\|,$$
which meas that $L_f$ is a bounded functional from Definition \ref{bounded}. By the Hahn-Banach type theorem in Theorem \ref{Hahn}, $L_f$ can be extended to  a linear functional $\widetilde{L}_f$ on $L^2(\Omega,\A,\varphi)$, and with
\begin{equation}\label{eq:6}
\begin{split}|\langle \widetilde{L}_f,g\rangle|_0\leq\sqrt{c^*}\|g\|,~\forall g\in L^2(\Omega,\A,\varphi),\end{split}
\end{equation}
where $\sqrt{c^*}=\sqrt{c}\cdot|e_0|_0$, {since} $|e_A|_0=2^{n/2}$, then $c^*=2^{n}c$ from \cite{c}. Now we are in the position to use the Riesz representation theorem for the operator $\widetilde{L}_f$. From Theorem \ref{Riesz}, there exists a $u\in L^2(\Omega,\A,\varphi)$ such that
\begin{equation}\label{eq:7}
\begin{split}\langle \widetilde{L}_f,g\rangle=(u,g)_\varphi,~\forall g\in L^2(\Omega,\A,\varphi).\end{split}
\end{equation}

For $\forall \alpha\in C^\infty_0(\Omega,\A)$, let $g=\overline{D}^*_\varphi\alpha$. Then
\begin{equation}\label{}
\begin{split}
(f,\alpha)_\varphi=&\langle  \widetilde{L}_f,\overline{D}^*_\varphi\alpha\rangle =(u,\overline{D}^*_\varphi\alpha)_\varphi=(\overline{D}u,\alpha)_\varphi,\nonumber
\end{split}
\end{equation}
which deduces that
$$\int_\Omega\bar{f}\alpha e^{-\varphi}dx=\int_\Omega\overline{(\overline{D}u)}{\alpha} e^{-\varphi}dx.$$ Conjugating both sides of above equation leads to
$$\int_\Omega\bar{\alpha}f \cdot e^{-\varphi}dx=\int_\Omega\bar{\alpha} (\overline{D})u  e^{-\varphi}dx.$$
Let $\alpha=\bar{\alpha}e^{\varphi}$,   it can be obtained that
$$\int_\Omega\alpha  fdx=\int_\Omega\alpha (\overline{D}u)dx,~\forall \alpha\in C^\infty_0(\Omega,\A).$$
Therefore, $$\overline{D}u=f$$ is proved from the definition of weak solutions.

Next, we give the bound for the norm of $u$. Let $g=u=\sum_{A}e_Au_A\in L^2(\Omega,\A,\varphi)$, from (\ref{eq:6}) and (\ref{eq:7}), we get that
\begin{equation}\label{eq:8}
\begin{split}|(u,u)_\varphi|_0\leq\sqrt{c^*}\|u\|.\end{split}
\end{equation}
On the other hand,
\begin{equation}
\begin{split}
|(u,u)_\varphi|_0^2=&\big|\int_\Omega\bar{u}ue^{-\varphi}dx\big|^2_0\\
=&~2^n\cdot\big[\int_\Omega\bar{u}ue^{-\varphi}dx\cdot\overline{\int_\Omega\bar{u}ue^{-\varphi}dx}\big]_0\\
=&~2^n\big[\int_\Omega(\sum\limits_Au^2_A+\sum\limits_{A\neq B}\bar{e}_Ae_Bu_Au_B)e^{-\varphi}dx\cdot\overline{\int_\Omega(\sum\limits_Au^2_A+\sum\limits_{A\neq B}\bar{e}_Ae_Bu_Au_B)e^{-\varphi}dx}\big]_0\\
=&~2^n\big[(\int_\Omega\sum\limits_Au^2_Ae^{-\varphi}dx)^2+(\int_\Omega\sum\limits_{A\neq B}u_Au_Be^{-\varphi}dx)^2\big],
\end{split}\nonumber
\end{equation}
and
\begin{equation}
\begin{split}
\|u\|^2=&~\int_\Omega|u|^2_0e^{-\varphi}dx
=2^n\int_\Omega[\bar{u}u]_0e^{-\varphi}dx
=2^n\int_\Omega\sum\limits_Au^2_A\cdot e^{-\varphi}dx
\end{split}\nonumber
\end{equation}
So we have $\|u\|^4=2^{2n}\cdot(\int_\Omega\sum\limits_Au^2_A\cdot e^{-\varphi}dx)^2$. Hence,
$$|(u,u)_\varphi|_0^2=2^n[(\int_\Omega\sum\limits_Au^2_A\cdot e^{-\varphi}dx)^2+(\int_\Omega\sum\limits_{A\neq B}u_Au_Be^{-\varphi}dx)^2]\geq 2^{-n}\|u\|^4.$$
Combining with (\ref{eq:8}), it is obtained that
$$\|u\|^2\leq 2^{n/2}|(u,u)_\varphi|_0\leq2^{n/2}\sqrt{c^*}\|u\|,$$ and
$$\|u\|^2\leq 2^{2n} {c}.$$
The proof is completed.
\end{proof}

\section{The proof of Theorem \ref{thm2}}
It should be noticed that inequality (\ref{eq:5}) in Theorem \ref{thm1} is related with $\alpha\in C^\infty_0(\Omega,\A)$. In the following, we will give another sufficient condition that has nothing to do with the space $C^\infty_0(\Omega,\A)$. First, we need to compute the norm of $\|\overline{D}^*_\varphi\alpha\|$ for any $\alpha\in C^\infty_0(\Omega,\A).$

\begin{equation}
\begin{split}
\|\overline{D}^*_\varphi\alpha\|^2=&\int_\Omega|\overline{D}^*_\varphi\alpha|^2_0e^{-\varphi}dx\\
=&\int_\Omega\langle\tau_{e_0},\overline{\overline{D}^*_\varphi\alpha}\cdot\overline{D}^*_\varphi\alpha\rangle e^{-\varphi}dx\\
=&\langle\tau_{e_0},\int_\Omega\overline{\overline{D}^*_\varphi\alpha}\cdot\overline{D}^*_\varphi\alpha e^{-\varphi}dx\rangle\\
=&\langle\tau_{e_0},(\overline{D}^*_\varphi\alpha,\overline{D}^*_\varphi\alpha)_\varphi\rangle\\
=&\langle\tau_{e_0},(\alpha,\overline{D}\overline{D}^*_\varphi\alpha)_\varphi\rangle\\
=&\langle\tau_{e_0},(\alpha,\overline{D}(\alpha (D\varphi)-D\alpha))_\varphi\rangle\\
=&\langle\tau_{e_0},(\alpha,\overline{D}\alpha (D\varphi)+\alpha\Delta\varphi-\Delta\alpha+\sum\limits^n_{j=1}(e_j\alpha-\alpha e_j)\frac{\partial}{\partial x_j}(D\varphi))_\varphi\rangle\\
=&\langle\tau_{e_0},(\alpha,\overline{D}^*_\varphi(\overline{D}\alpha)+\alpha\Delta\varphi+\sum\limits^n_{j=1}(e_j\alpha-\alpha e_j)\frac{\partial}{\partial x_j}(D\varphi))_\varphi\rangle\\
=&\langle\tau_{e_0},(\alpha,\overline{D}^*_\varphi(\overline{D}\alpha))_\varphi+(\alpha,\alpha\Delta\varphi)_\varphi+(\alpha,\sum\limits^n_{j=1}(e_j\alpha-\alpha e_j)\frac{\partial}{\partial x_j}(D\varphi))_\varphi\rangle\\
=&\langle\tau_{e_0},(\alpha,\overline{D}^*_\varphi(\overline{D}\alpha))_\varphi\rangle+\langle\tau_{e_0},(\alpha,\alpha\Delta\varphi)_\varphi\rangle+\langle\tau_{e_0},(\alpha,\sum\limits^n_{j=1}(e_j\alpha-\alpha e_j)\frac{\partial}{\partial x_j}(D\varphi))_\varphi\rangle\\
=&I_1+I_2+I_3,\nonumber
\end{split}
\end{equation}
where
\begin{equation}
\begin{split}
I_1=&\langle\tau_{e_0},(\alpha,\overline{D}^*_\varphi(\overline{D}\alpha))_\varphi\rangle=\langle\tau_{e_0},(\overline{D}\alpha,\overline{D}\alpha)_\varphi\rangle=\|\overline{D}\alpha\|^2,\\
I_2=&\langle\tau_{e_0},(\alpha,\alpha\Delta\varphi)_\varphi\rangle=\int_{\Omega}|\alpha|^2_0\Delta\varphi e^{-\varphi}dx, \nonumber
\end{split}
\end{equation}
and
\begin{equation}
\begin{split}
I_3=&\langle\tau_{e_0},(\alpha,\sum\limits^n_{j=1}(e_j\alpha-\alpha e_j)\frac{\partial}{\partial x_j}(D\varphi))_\varphi\rangle\\
=&\langle\tau_{e_0},(\alpha,\sum\limits^n_{j=1}(e_j\alpha-\alpha e_j)\frac{\partial}{\partial x_j}(\sum_{i=0}^{n}\bar{e}_i\frac{\partial \varphi}{\partial x_i}))_\varphi\rangle\\
=&\langle\tau_{e_0},(\alpha,\sum\limits^n_{j=1}\sum_{i=0}^{n}(e_j\alpha \bar{e}_i-\alpha e_j \bar{e}_i)\frac{\partial^2 \varphi}{\partial x_j\partial x_i})_\varphi\rangle\\
=&\langle\tau_{e_0},\int_\Omega \bar{\alpha}\sum\limits^n_{j=1}\sum_{i=0}^{n}(e_j\alpha \bar{e}_i-\alpha e_j \bar{e}_i)\frac{\partial^2 \varphi}{\partial x_j\partial x_i} e^{-\varphi}dx\rangle\\
=&\int_\Omega\langle\tau_{e_0}, \bar{\alpha}\sum\limits^n_{j=1}\sum_{i=0}^{n}(e_j\alpha \bar{e}_i-\alpha e_j \bar{e}_i)\frac{\partial^2 \varphi}{\partial x_j\partial x_i}\rangle e^{-\varphi}dx.
\end{split}\nonumber
\end{equation}
{\bf It should be noticed that if $n=1$, i.e., the space $\R^2$ is considered, then $I_3=0.$}

Since for $1\leq i,j\leq n$ and $i\neq j$, $e_j \bar{e}_i=-e_j {e}_i=e_i{e}_j=- {e}_i\bar{e}_j$. For simplicity, let
\begin{equation}
\begin{split}
I_4=&\langle\tau_{e_0},\bar{\alpha}\sum\limits^n_{j=1}\sum\limits_{i=0}^{n}(e_j\alpha \bar{e}_i-\alpha e_j \bar{e}_i)\frac{\partial^2 \varphi}{\partial x_j\partial x_i}\rangle\\
=&\langle\tau_{e_0},\sum\limits^n_{j=1}\sum\limits_{i=1}^{n}(\bar{\alpha}e_j\alpha \bar{e}_i-\bar{\alpha}\alpha e_j \bar{e}_i)\frac{\partial^2 \varphi}{\partial x_j\partial x_i}\rangle+\langle\tau_{e_0},\sum\limits^n_{j=1}(\bar{\alpha}e_j\alpha \bar{e}_0-\bar{\alpha}\alpha e_j \bar{e}_0)\frac{\partial^2 \varphi}{\partial x_j\partial x_0}\rangle\\
=&\langle\tau_{e_0},\sum\limits_{i=1}^{n}(\bar{\alpha}e_i\alpha \bar{e}_i-\bar{\alpha}\alpha e_i \bar{e}_i)\frac{\partial^2 \varphi}{\partial x^2_i}\rangle+\langle\tau_{e_0},\sum\limits^n_{j\neq i}(\bar{\alpha}e_j\alpha \bar{e}_i)\frac{\partial^2 \varphi}{\partial x_j\partial x_i}\rangle\\
&+\langle\tau_{e_0},\sum\limits^n_{j=1}(\bar{\alpha}e_j\alpha \bar{e}_0-\bar{\alpha}\alpha e_j \bar{e}_0)\frac{\partial^2 \varphi}{\partial x_j\partial x_0}\rangle\\
=&\langle\tau_{e_0},\sum\limits_{i=1}^{n}(\bar{\alpha}e_i\alpha \bar{e}_i-\bar{\alpha}\alpha )\frac{\partial^2 \varphi}{\partial x^2_i}\rangle+\langle\tau_{e_0},\sum\limits^n_{j\neq i}(\bar{\alpha}e_j\alpha \bar{e}_i)\frac{\partial^2 \varphi}{\partial x_j\partial x_i}\rangle\\
&+\langle\tau_{e_0},\sum\limits^n_{j=1}(\bar{\alpha}e_j\alpha \bar{e}_0-\bar{\alpha}\alpha e_j \bar{e}_0)\frac{\partial^2 \varphi}{\partial x_j\partial x_0}\rangle\\
=&I_5+I_6+I_7.
\end{split}\nonumber
\end{equation}
Assume $\alpha=\sum\limits_A\alpha_Ae_A\in \A,~\bar{\alpha}=\sum\limits_A\alpha_A\bar{e}_A$, then for any $1\leq i\leq n,$
\begin{equation}
\begin{split}
\bar{\alpha}e_i\alpha\bar{e}_i=&~\sum\limits_A\alpha_A\bar{e}_Ae_i\cdot\sum\limits_A\alpha_Ae_A\bar{e}_i\\
=&~\sum\limits_A(-1)^{\frac{|A|(|A|+1)}{2}}\alpha_Ae_Ae_i\cdot\sum\limits_A(-1)\alpha_Ae_Ae_i
\end{split}\nonumber
\end{equation}
Therefore
\begin{equation}
\begin{split}
I_5=&\langle\tau_{e_0},\sum\limits_{i=1}^{n}(\bar{\alpha}e_i\alpha \bar{e}_i-\bar{\alpha}\alpha )\frac{\partial^2 \varphi}{\partial x^2_i}\rangle\\
=&\langle\tau_{e_0},\sum\limits_{i=1}^{n}(\bar{\alpha}e_i\alpha \bar{e}_i)\frac{\partial^2 \varphi}{\partial x^2_i}\rangle-\langle\tau_{e_0},\sum\limits_{i=1}^{n}(\bar{\alpha}\alpha )\frac{\partial^2 \varphi}{\partial x^2_i}\rangle\\
=&\langle\tau_{e_0},\sum\limits_{i=1}^{n}(\sum\limits_A(-1)^{\frac{|A|(|A|+1)}{2}}\alpha_Ae_Ae_i\cdot\sum\limits_A(-1)\alpha_Ae_Ae_i)\frac{\partial^2 \varphi}{\partial x^2_i}\rangle-\langle\tau_{e_0},\sum\limits_{i=1}^{n}(\bar{\alpha}\alpha )\frac{\partial^2 \varphi}{\partial x^2_i}\rangle\\
=&2^n\sum\limits_{i=1}^{n}(\sum\limits_A(-1)^{\frac{|A|(|A|+1)}{2}+1}\alpha_A^2 e_Ae_ie_Ae_i)\frac{\partial^2 \varphi}{\partial x^2_i}-\sum\limits_{i=1}^{n}|\alpha|^2_0\frac{\partial^2 \varphi}{\partial x^2_i}\\
=&2^n\sum\limits_{i=1}^{n}(\sum\limits_{i\not\in A}(-1)^{\frac{|A|(|A|+1)}{2}+1}\alpha^2_A\cdot\overline{e_Ae_i}\cdot e_Ae_i\cdot
(-1)^{\frac{(|A|+1)(|A|+2)}{2}}\\
&+\sum\limits_{i\in A}(-1)^{\frac{|A|(|A|+1)}{2}+1}\cdot\alpha^2_A\cdot\overline{e_{A-{i}}}\cdot
e_{A-{i}}\cdot(-1)^{\frac{(|A|-1)(|A|)}{2}})\frac{\partial^2 \varphi}{\partial x^2_i}-\sum\limits_{i=1}^{n}|\alpha|^2_0\frac{\partial^2 \varphi}{\partial x^2_i}\\
=&2^n\sum\limits_{i=1}^{n}(\sum\limits_{i\not\in A}(-1)^{\frac{|A|(|A|+1)}{2}+1+\frac{(|A|+1)(|A|+2)}{2}}\cdot\alpha^2_A\\
&+\sum\limits_{i\in A}(-1)^{\frac{|A|(|A|+1)}{2}+1+\frac{(|A|-1)(|A|)}{2}}\cdot\alpha^2_A)\frac{\partial^2 \varphi}{\partial x^2_i}-\sum\limits_{i=1}^{n}|\alpha|^2_0\frac{\partial^2 \varphi}{\partial x^2_i}\\
=&2^n\sum\limits_{i=1}^{n}(\sum\limits_{i\not\in A}(-1)^{|A|^2}\cdot\alpha^2_A+\sum\limits_{i\in A}(-1)^{|A|^2+1}\cdot\alpha^2_A)\frac{\partial^2 \varphi}{\partial x^2_i}-\sum\limits_{i=1}^{n}|\alpha|^2_0\frac{\partial^2 \varphi}{\partial x^2_i}\\
=&2^n\sum\limits_{i=1}^{n}(\sum\limits_{i\not\in A,|A|^2 ~\mbox{is odd}}(-2)\alpha^2_A+\sum\limits_{i\in A,|A|^2 ~\mbox{is even}}(-2)\alpha^2_A)\frac{\partial^2 \varphi}{\partial x^2_i}\\
=&-2^{n+1}\sum\limits_{i=1}^{n}(\sum\limits_{i\not\in A,|A|^2 ~\mbox{is odd}}\alpha^2_A+\sum\limits_{i\in A,|A|^2 ~\mbox{is even}}\alpha^2_A)\frac{\partial^2 \varphi}{\partial x^2_i}.
\end{split}
\end{equation}
To consider $I_7$, we first study $\bar{\alpha}e_j\alpha$ for any $1\leq j\leq n$. Without loss of generality, let $e_j=e_1,~\bar{\alpha}=\sum\limits_A\alpha_A\bar{e}_A,~
\alpha=\sum\limits_A\alpha_Ae_A$. Then $\bar{\alpha}e_1\alpha=(\sum\limits_A\alpha_A\bar{e}_A)e_1(\sum\limits_A\alpha_Ae_A)$.

When $e_A=e_1e_{h_2}e_{h_3}\cdots e_{h_r}$,~where $1<h_2<h_3<\cdots<h_r$ and $1<r\leq n.$
\begin{equation}\label{16}
\begin{split}
\alpha_A\bar{e}_Ae_1=&\alpha_{1h_2\cdots h_r}(-1)^{\frac{r(r+1)}{2}}\cdot e_1e_{h_2}e_{h_3}\cdots e_{h_r}\cdot e_1
\\=&\alpha_{1h_2\cdots h_r}(-1)^{\frac{r(r+1)}{2}+r}e_{h_2}e_{h_3}\cdots e_{h_r}\\
\alpha_Ae_Ae_1=&\alpha_{1h_2\cdots h_r}e_1e_{h_2}\cdots e_{h_r}\cdot e_1=\alpha_{1h_2\cdots h_r}(-1)^re_{h_2}\cdots e_{h_r}.
\end{split}
\end{equation}

When $e_A=e_1$,
\begin{equation}\label{165}
\begin{split}
\alpha_A\bar{e}_Ae_1=&\alpha_{1}\\
\alpha_Ae_Ae_1=&-\alpha_{1}.
\end{split}
\end{equation}

When $e_A=e_{h_2}e_{h_3}\cdots e_{h_r}$,~where $1<h_2<h_3<\cdots<h_r$ and $1<r\leq n.$
\begin{equation}\label{17}
\begin{split}
\alpha_A\bar{e}_Ae_1=&\alpha_{h_2\cdots h_r}(-1)^{\frac{(r-1)(r)}{2}}\cdot e_{h_2}e_{h_3}\cdots e_{h_r}\cdot e_1
\\=&\alpha_{h_2\cdots h_r}(-1)^{\frac{(r-1)(r)}{2}+r-1}e_1e_{h_2}\cdots e_{h_r}\\
\alpha_Ae_Ae_1=&\alpha_{h_2\cdots h_r}e_{h_2}\cdots e_{h_r}\cdot e_1=\alpha_{h_2\cdots h_r}(-1)^{r-1}e_1e_{h_2}\cdots e_{h_r}.
\end{split}
\end{equation}

When $e_A=e_0$,
\begin{equation}\label{175}
\begin{split}
\alpha_A\bar{e}_Ae_1=&\alpha_{0}e_1\\
\alpha_Ae_Ae_1=&\alpha_{0}e_1.
\end{split}
\end{equation}
To compute $I_7$, one needs to know the coefficient for $e_0$ of $\bar{\alpha}e_1\alpha-\bar{\alpha}\alpha e_1$. It means that we should find out the corresponding terms of $e_1e_{h_2}e_{h_3}\cdots e_{h_r}$ and $e_{h_2}\cdots e_{h_r}$ in $\bar{\alpha}e_1$ and $\alpha$, in $\bar{\alpha}$ and $\alpha e_1$.

{\bf Case a1.} For $\bar{\alpha}e_1\alpha$, from (\ref{17}), the corresponding terms of $e_1e_{h_2}e_{h_3}\cdots e_{h_r}$ with $1<h_2<h_3<\cdots<h_r$ and $1<r\leq n$ in $\bar{\alpha}e_1=(\sum\limits_A\alpha_A\bar{e}_A)e_1$ and $\alpha=\sum\limits_A\alpha_Ae_A$ are $\alpha_{h_2\cdots h_r}(-1)^{\frac{(r-1)(r)}{2}+r-1}e_1e_{h_2}\cdots e_{h_r}$ and $\alpha_{1h_2\cdots h_r}e_1e_{h_2}\cdots e_{h_r}$, respectively. Multiplying these terms leads to
\begin{equation}\label{18}
\begin{split}
(-1)&^{\frac{(r-1)(r)}{2}+r-1}e_1e_{h_2}\cdots e_{h_r}\cdot e_1e_{h_2}\cdots e_{h_r}\cdot\alpha_{1h_2\cdots h_r}\cdot\alpha_{h_2\cdots h_r}\\
=&~(-1)^{\frac{(r-1)(r)}{2}+r-1}(-1)^{\frac{(r)(r+1)}{2}}\cdot\overline{e_1\cdots e_{h_r}}\cdot e_1e_{h_2}\cdots e_{h_r} \cdot\alpha_{1h_2\cdots h_r}\alpha_{h_2\cdots h_r}\\
=&~(-1)^{\frac{(r)(r+1)}{2}+r-1+\frac{(r-1)(r)}{2}}\cdot\alpha_{1h_2\cdots h_r}\alpha_{h_2\cdots h_r}.
\end{split}
\end{equation}
On the other hand, for $\bar{\alpha}e_1\alpha$, from (\ref{16}), the corresponding terms of $e_{h_2}e_{h_3}\cdots e_{h_r}$ with $1<h_2<h_3<\cdots<h_r$ and $1<r\leq n$ in $\bar{\alpha}e_1$ and $\alpha$ are $\alpha_{1h_2\cdots h_r}(-1)^{\frac{r(r+1)}{2}+r}e_{h_2}e_{h_3}\cdots e_{h_r}$ and $\alpha_{h_2\cdots h_r}e_{h_2}\cdots e_{h_r}$, respectively. Multiplying these terms leads to
\begin{equation}\label{19}
\begin{split}
(-1)&^{\frac{(r)(r+1)}{2}+r}e_{h_2\cdots h_r}\cdot e_{h_2\cdots h_r}\cdot\alpha_{1h_2\cdots h_r}
\cdot\alpha_{h_2\cdots h_r}\\
=&~(-1)^{\frac{(r)(r+1)}{2}+r}(-1)^{\frac{(r-1)(r)}{2}}\cdot\overline{e_{h_2\cdots h_r}}\cdot e_{h_2\cdots h_r} \cdot\alpha_{1h_2\cdots h_r}\alpha_{h_2\cdots h_r}\\
=&~(-1)^{\frac{(r)(r+1)}{2}+r+\frac{(r-1)(r)}{2}}\cdot\alpha_{1h_2\cdots h_r}\alpha_{h_2\cdots h_r}.
\end{split}
\end{equation}
From (\ref{18}) and (\ref{19}), these two terms vanish.

{\bf Case a2.} For $\bar{\alpha}e_1\alpha$, from (\ref{175}), the corresponding terms of $e_1$ in $\bar{\alpha}e_1$ and $\alpha$ are $\alpha_{0}e_1$ and $\alpha_{1}e_1$, respectively. Multiplying these terms leads to
\begin{equation}\label{185}
\begin{split}
\alpha_{0}e_1\alpha_{1}e_1=-\alpha_{0}\alpha_{1}.
\end{split}
\end{equation}
On the other hand, for $\bar{\alpha}e_1\alpha$, from (\ref{165}), the corresponding terms of $e_{0}$ in $\bar{\alpha}e_1$ and $\alpha$ are $\alpha_{1}$ and $\alpha_{0}$, respectively. Multiplying these terms leads to $\alpha_{0}\alpha_{1}$. Combining with (\ref{185}), these two terms also vanish.

From Cases a1 and a2, one can obtain that the coefficient for $e_0$ of $\bar{\alpha}e_1\alpha$ equals zero, i.e.,
\begin{equation}\label{23}
\begin{split}\langle\tau_{e_0},\sum\limits^n_{j=1}(\bar{\alpha}e_j\alpha \bar{e}_0)\frac{\partial^2 \varphi}{\partial x_j\partial x_0}\rangle=0.\end{split}
\end{equation}

%%%%%%%%%%%%
{\bf Case b1.} For $\bar{\alpha}\alpha e_1$, from (\ref{17}), the corresponding terms of $e_1e_{h_2}e_{h_3}\cdots e_{h_r}$ with $1<h_2<h_3<\cdots<h_r$ and $1<r\leq n$ in ${\alpha}e_1=(\sum\limits_A\alpha_A{e}_A)e_1$ and $\bar{\alpha}=\sum\limits_A\alpha_A\bar{e}_A$ are $\alpha_{h_2\cdots h_r}(-1)^{r-1}e_1e_{h_2}\cdots e_{h_r}$ and $\alpha_{1h_2\cdots h_r}\overline{e_1e_{h_2}\cdots e_{h_r}}$, respectively. Multiplying these terms leads to
\begin{equation}\label{25}
\begin{split}
(\alpha_{1h_2\cdots h_r}&\overline{e_1e_{h_2}\cdots e_{h_r}})\cdot(\alpha_{h_2\cdots h_r}e_{h_2}\cdots e_{h_r}\cdot e_1)\\
=&~(\alpha_{1h_2\cdots h_r}\overline{e_1e_{h_2}\cdots e_{h_r}})\cdot((-1)^{r-1}e_1e_{h_2}\cdots e_{h_r}\cdot \alpha_{h_2\cdots h_r})\\
=&~(-1)^{r-1}\alpha_{1h_2\cdots h_r}\cdot\alpha_{h_2\cdots h_r}.
\end{split}
\end{equation}
On the other hand, for $\bar{\alpha}\alpha e_1$, from (\ref{16}), the corresponding terms of $e_{h_2}e_{h_3}\cdots e_{h_r}$ with $1<h_2<h_3<\cdots<h_r$ and $1<r\leq n$ in ${\alpha}e_1$ and $\bar{\alpha}$ are $\alpha_{1h_2\cdots h_r}(-1)^re_{h_2}\cdots e_{h_r}$ and $\alpha_{h_2\cdots h_r}\overline{e_{h_2}\cdots e_{h_r}}$, respectively. Multiplying these terms leads to
\begin{equation}\label{26}
\begin{split}
(\alpha_{h_2\cdots h_r}&\overline{e_{h_2}\cdots e_{h_r}})\cdot(\alpha_{1h_2\cdots h_r}e_1\cdots e_{h_r}\cdot e_1)\\
=&~(\alpha_{h_2\cdots h_r}\overline{e_{h_2}\cdots e_{h_r}})\cdot((-1)^re_{h_2}\cdots e_{h_r}\cdot \alpha_{1h_2\cdots h_r})\\
=&~(-1)^r\alpha_{h_2\cdots h_r}\cdot\alpha_{1h_2\cdots h_r}.
\end{split}
\end{equation}
From (\ref{25}) and (\ref{26}), these two terms vanish.

{\bf Case b2.} For $\bar{\alpha}\alpha e_1$, from (\ref{175}), the corresponding terms of $e_1$ in ${\alpha}e_1$ and $\bar{\alpha}$ are $\alpha_{0}e_1$ and $\alpha_{1}\bar{e}_1$, respectively. Multiplying these terms leads to
\begin{equation}\label{1855}
\begin{split}
\alpha_{0}e_1\alpha_{1}\bar{e}_1=\alpha_{0}\alpha_{1}.
\end{split}
\end{equation}
On the other hand, for $\bar{\alpha}\alpha e_1$, from (\ref{165}), the corresponding terms of $e_{0}$ in ${\alpha}e_1$ and $\bar{\alpha}$ are $-\alpha_{1}$ and $\alpha_{0}$, respectively. Multiplying these terms leads to $-\alpha_{0}\alpha_{1}$. Combining with (\ref{1855}), these two terms also cancel.

From Cases b1 and b2, one can obtain that the coefficient for $e_0$ of $\bar{\alpha}e_1\alpha$ equals zero, i.e.,
\begin{equation}\label{24}
\begin{split}\langle\tau_{e_0},\sum\limits^n_{j=1}(\bar{\alpha}\alpha e_j\bar{e}_0)\frac{\partial^2 \varphi}{\partial x_j\partial x_0}\rangle=0.\end{split}
\end{equation}
{\bf Thus, $I_7=0$ from (\ref{23}) and (\ref{24}).}

To compute $I_6$, i.e., to get $[\bar{\alpha}e_i\alpha \bar{e}_j]_0$ for $i\neq j$, similar with the analysis of $I_7$, we should divide the vectors in $\bar{\alpha}e_i$ and $\alpha \bar{e}_j$ into four
cases.

{\bf Case c1.} $i\in A,~j\not\in A$ for $e_A$ in $\bar{\alpha}$ and $i\not\in B,~j\in B$ for $e_B$ in ${\alpha}$ with $A-{i}=B-{j}$.

For this case, firstly, we assume $e_A=e_{h_1\cdots h_{p(i)}\cdots h_r}$ and $h_{p(i)}=i$, $e_B=e_{h_1\cdots h_{p(j)}\cdots h_r}$ and $h_{p(j)}=j$.
We have
\begin{equation}\label{}
\begin{split}
\alpha_A\bar{e}_Ae_i=&\alpha_{A}(-1)^{\frac{r(r+1)}{2}}\cdot e_{h_1}\cdots e_i\cdots e_{h_r}\cdot e_i\\
=&\alpha_{A}(-1)^{\frac{r(r+1)}{2}+r-p(i)} e_{h_1}\cdots e_i^2\cdots e_{h_r},\\
=&\alpha_{A}(-1)^{\frac{r(r+1)}{2}+r-p(i)+1} e_{A-{i}},\\
\alpha_B{e}_B\bar{e}_j=&\alpha_{B} e_{h_1}\cdots e_j\cdots e_{h_r}\cdot \bar{e}_j\\
=&\alpha_{B}(-1)^{r-p(j)} e_{h_1}\cdots e_j\bar{e}_j\cdots e_{h_r},\\
=&\alpha_{B}(-1)^{r-p(j)} e_{B-{j}}.
\end{split}\nonumber
\end{equation}
Then
\begin{equation}\label{}
\begin{split}
\alpha_A\bar{e}_Ae_i\alpha_B{e}_B\bar{e}_j=&\alpha_{A}(-1)^{\frac{r(r+1)}{2}+r-p(i)+1} e_{A-{i}}\alpha_{B}(-1)^{r-p(j)} e_{B-{j}}\\
=&\alpha_{A}\alpha_{B}(-1)^{\frac{r(r+1)}{2}+r-p(i)+1+r-p(j)+\frac{r(r-1)}{2}} \overline{e_{A-{i}}} e_{B-{j}}\\
=&\alpha_{A}\alpha_{B}(-1)^{r^2+1-p(i)-p(j)}.
\end{split}
\end{equation}

{\bf Case c2.} $i\not\in A,~j\in A$ for $e_A$ in $\bar{\alpha}$ and $i\in B,~j\not\in B$ for $e_B$ in ${\alpha}$ with $A+{i}=B+{j}$.

We assume $e_A=e_{h_1\cdots h_{p(j)}\cdots h_r}$ and $h_{p(j)}=j$, $e_B=e_{h_1\cdots h_{p(i)}\cdots h_r}$ and $h_{p(i)}=i$.
We have
\begin{equation}\label{}
\begin{split}
\alpha_A\bar{e}_Ae_i=&\alpha_{A}(-1)^{\frac{r(r+1)}{2}}\cdot e_{h_1}\cdots e_j\cdots e_{h_r}\cdot e_i,\\
\alpha_B{e}_B\bar{e}_j=&\alpha_{B} e_{h_1}\cdots e_i\cdots e_{h_r}\cdot \bar{e}_j\\
=&-\alpha_{B} e_{h_1}\cdots e_i\cdots e_{h_r}\cdot{e}_j\\
=&\alpha_{B} e_{h_1}\cdots e_j\cdots e_{h_r}\cdot e_i.
\end{split}\nonumber
\end{equation}
Then
\begin{equation}\label{}
\begin{split}
\alpha_A\bar{e}_Ae_i\alpha_B{e}_B\bar{e}_j=&\alpha_{A}(-1)^{\frac{r(r+1)}{2}}\cdot e_{h_1}\cdots e_j\cdots e_{h_r}\cdot e_i\alpha_{B} e_{h_1}\cdots e_j\cdots e_{h_r}\cdot e_i\\
=&\alpha_{A}\alpha_{B}(-1)^{\frac{r(r+1)}{2}+\frac{(r+1)(r+2)}{2}} \overline{e_{h_1}\cdots e_j\cdots e_{h_r}\cdot e_i} e_{h_1}\cdots e_j\cdots e_{h_r}\cdot e_i\\
=&\alpha_{A}\alpha_{B}(-1)^{r^2+1}.
\end{split}\nonumber
\end{equation}

{\bf Case c3.} $i\in A,~j\in A$ for $e_A$ in $\bar{\alpha}$ and $i\not\in B,~j\not\in B$ for $e_B$ in ${\alpha}$ with $A-{i}=B+{j}$.

For this case, we assume $e_A=e_{h_1\cdots h_{p(i)}\cdots h_{p(j)}\cdots h_{r+2}}$ with $h_{p(i)}=i,~ h_{p(j)}=j$. Without loss of generality, we assume $i<j$. Furthermore, let $e_B=e_{h_1\cdots h_r}$.
We have
\begin{equation}\label{}
\begin{split}
\alpha_A\bar{e}_Ae_i=&\alpha_{A}(-1)^{\frac{(r+2)(r+3)}{2}}\cdot e_{h_1}\cdots e_i\cdots e_j\cdots e_{h_{r+2}}\cdot e_i\\
=&\alpha_{A}(-1)^{\frac{(r+2)(r+3)}{2}+r+2-h(i)}\cdot e_{h_1}\cdots e_j\cdots e_{h_{r+2}}\cdot e^2_i\\
=&\alpha_{A}(-1)^{\frac{(r+2)(r+3)}{2}+r+1-h(i)}\cdot e_{h_1}\cdots e_j\cdots e_{h_{r+2}}\\
=&\alpha_{A}(-1)^{\frac{(r+2)(r+3)}{2}+r+1-h(i)+r+2-h(j)}\cdot e_{h_1}\cdots e_{h_{r+2}}\cdot e_j,\\
\alpha_B{e}_B\bar{e}_j=&\alpha_{B} e_{h_1}\cdots e_{h_{r}}\cdot \bar{e}_j\\
=&-\alpha_{B} e_{h_1}\cdots e_{h_{r}}\cdot{e}_j.
\end{split}\nonumber
\end{equation}
Then
\begin{equation}\label{}
\begin{split}
\alpha_A\bar{e}_Ae_i\alpha_B{e}_B\bar{e}_j=&\alpha_{A}(-1)^{\frac{(r+2)(r+3)}{2}+r+1-h(i)+r+2-h(j)}\cdot e_{h_1}\cdots e_{h_{r+2}}\cdot e_j (-1)\alpha_{B} e_{h_1}\cdots e_{h_{r}}\cdot{e}_j \\
=&\alpha_{A}\alpha_{B} (-1)^{\frac{(r+2)(r+3)}{2}-h(i)-h(j)}\cdot e_{h_1}\cdots e_{h_{r+2}} \cdot e_j  e_{h_1}\cdots e_{h_r}\cdot e_j\\
=&\alpha_{A}\alpha_{B} (-1)^{\frac{(r+2)(r+3)}{2}-h(i)-h(j)+\frac{(r+1)(r+2)}{2}}\cdot \overline{e_{h_1}\cdots e_{h_{r+2}} \cdot e_j}  e_{h_1}\cdots e_{h_r}\cdot e_j\\
=&\alpha_{A}\alpha_{B} (-1)^{r^2-h(j)-h(i)}.
\end{split}\nonumber
\end{equation}

{\bf Case c4.} $i\not\in A,~j\not\in A$ for $e_A$ in $\bar{\alpha}$ and $i\in B,~j\in B$ for $e_B$ in ${\alpha}$ with $A+{i}=B-{j}$.

For this case, we assume $e_A=e_{h_1\cdots h_r}$, $e_B=e_{h_1\cdots h_{p(i)}\cdots h_{p(j)}\cdots h_{r+2}}$ with $h_{p(i)}=i,~ h_{p(j)}=j$ and $i<j$.
We have
\begin{equation}\label{}
\begin{split}
\alpha_A\bar{e}_Ae_i=&\alpha_{A}(-1)^{\frac{r(r+1)}{2}}\cdot e_{h_1}\cdots e_{h_r}\cdot e_i,\\
\alpha_B{e}_B\bar{e}_j=&\alpha_{B} e_{h_1}\cdots e_i\cdots e_j\cdots e_{h_{r+2}}\cdot \bar{e}_j\\
=&\alpha_{B} (-1)^{r+2-h(j)}\cdot e_{h_1}\cdots e_i\cdots e_{h_{r+2}}\cdot e_j \bar{e}_j\\
=&\alpha_{B} (-1)^{r+2-h(j)+r+2-h(i)-1}\cdot e_{h_1}\cdots e_{h_{r+2}}\cdot e_i \\
=&\alpha_{B} (-1)^{1-h(j)-h(i)}\cdot e_{h_1}\cdots e_{h_{r+2}}\cdot e_i \\
\end{split}\nonumber
\end{equation}
Then
\begin{equation}\label{}
\begin{split}
\alpha_A\bar{e}_Ae_i\alpha_B{e}_B\bar{e}_j=&\alpha_{A}(-1)^{\frac{r(r+1)}{2}}\cdot e_{h_1}\cdots e_{h_r}\cdot e_i\alpha_{B} (-1)^{1-h(j)-h(i)}\cdot e_{h_1}\cdots e_{h_{r+2}}\cdot e_i \\
=&\alpha_{A}\alpha_{B} (-1)^{\frac{r(r+1)}{2}+1-h(j)-h(i)}\cdot e_{h_1}\cdots e_{h_r}\cdot e_i\cdot e_{h_1}\cdots e_{h_{r+2}}\cdot e_i \\
=&\alpha_{A}\alpha_{B} (-1)^{\frac{r(r+1)}{2}+1-h(j)-h(i)+\frac{(r+1)(r+2)}{2}}\cdot \overline{e_{h_1}\cdots e_{h_r}\cdot e_i}\cdot e_{h_1}\cdots e_{h_{r+2}}\cdot e_i \\
=&\alpha_{A}\alpha_{B} (-1)^{r^2-h(j)-h(i)}.
\end{split}\nonumber
\end{equation}

Combining cases c1-c4, we have
\begin{equation}\label{}
\begin{split}
I_6=&\langle\tau_{e_0},\sum\limits^n_{j\neq i}(\bar{\alpha}e_j\alpha \bar{e}_i)\frac{\partial^2 \varphi}{\partial x_j\partial x_i}\rangle\\
   =&\langle\tau_{e_0},\sum\limits^n_{j\neq i}\big((\sum_{A}\bar{e_A}\alpha_A)e_j(\sum_{B}{e_B}\alpha_B) \bar{e}_i\big)\frac{\partial^2 \varphi}{\partial x_j\partial x_i}\rangle\\
   =&\langle\tau_{e_0},\sum\limits^n_{j\neq i}\big((\sum_{A}\bar{e_A}\alpha_A)e_i(\sum_{B}{e_B}\alpha_B) \bar{e}_j\big)\frac{\partial^2 \varphi}{\partial x_i\partial x_j}\rangle\\
   =&\sum\limits^n_{j\neq i}\langle\tau_{e_0},(\sum_{A}\bar{e_A}\alpha_A)e_i(\sum_{B}{e_B}\alpha_B) \bar{e}_j\rangle\frac{\partial^2 \varphi}{\partial x_i\partial x_j}\\
   =&\sum\limits^n_{j\neq i}\langle\tau_{e_0},(\sum_{A}\bar{e_A}\alpha_A)e_i(\sum_{B}{e_B}\alpha_B) \bar{e}_j\rangle\frac{\partial^2 \varphi}{\partial x_i\partial x_j}\\
   =&2^n\sum\limits^n_{j\neq i}\Big(\sum_{i\in A,~j\not\in A;A-{i}=B-{j}}\alpha_{A}\alpha_{B}(-1)^{r^2+1-p(i)-p(j)}\\
   &+\sum_{i\not\in A,~j\in A;A+{i}=B+{j}}\alpha_{A}\alpha_{B}(-1)^{r^2+1}\\
   &+\sum_{i\in A,~j\in A;A-{i}=B+{j}}\alpha_{A}\alpha_{B} (-1)^{r^2-h(j)-h(i)}\\
   &+\sum_{i\not\in A,~j\not\in A;A+{i}=B-{j}}\alpha_{A}\alpha_{B} (-1)^{r^2-h(j)-h(i)}\Big)\frac{\partial^2 \varphi}{\partial x_i\partial x_j}.
\end{split}\nonumber
\end{equation}

In all,
\begin{equation}
\begin{split}
I_3=&\int_\Omega I_4e^{-\varphi}dx\\
=&\int_\Omega (I_5+I_6+I_7)e^{-\varphi}dx\\
=&-2^{n+1}\int_\Omega \sum\limits_{i=1}^{n}(\sum\limits_{i\not\in A,|A|^2 ~\mbox{is odd}}\alpha^2_A+\sum\limits_{i\in A,|A|^2 ~\mbox{is even}}\alpha^2_A)\frac{\partial^2 \varphi}{\partial x^2_i}e^{-\varphi}dx\\
&+2^n\int_\Omega \sum\limits^n_{j\neq i}\Big(\sum_{i\in A,~j\not\in A;A-{i}=B-{j}}\alpha_{A}\alpha_{B}(-1)^{r^2+1-p(i)-p(j)}\\
   &+\sum_{i\not\in A,~j\in A;A+{i}=B+{j}}\alpha_{A}\alpha_{B}(-1)^{r^2+1}\\
   &+\sum_{i\in A,~j\in A;A-{i}=B+{j}}\alpha_{A}\alpha_{B} (-1)^{r^2-h(j)-h(i)}\\
   &+\sum_{i\not\in A,~j\not\in A;A+{i}=B-{j}}\alpha_{A}\alpha_{B} (-1)^{r^2-h(j)-h(i)}\Big)\frac{\partial^2 \varphi}{\partial x_i\partial x_j} e^{-\varphi}dx.
\end{split}\nonumber
\end{equation}
Then
\begin{equation}\label{38}
\begin{split}
\|\overline{D}^*_\varphi\alpha\|^2=\|\overline{D}\alpha\|^2+\int_\Omega|\alpha|^2_0\Delta\varphi e^{-\varphi}dx
+I_3.
\end{split}
\end{equation}
If $\frac{\partial^2 \varphi}{\partial x_j\partial x_i}=0,~i\neq j,~1\leq i,j\leq n$ and $\frac{\partial^2 \varphi}{\partial x^2_i}\leq 0,~1\leq i\leq n$,  we have $I_3\geq 0$, and
$$\|\overline{D}^*_\varphi\alpha\|^2\geq \int_\Omega|\alpha|^2_0\Delta\varphi e^{-\varphi}dx.$$
With the above analysis, we can prove Theorem \ref{thm2} easily.
\begin{proof}
It is sufficient to prove the theorem if condition (\ref{eq:5}) in Theorem \ref{thm1} is presented. By Cauchy-Schwarz inequality in Proposition \ref{prop1}, we have for any $\alpha\in C^\infty_0(\Omega,\A)$ that
\begin{equation}
\begin{split}
|({f},\alpha)_\varphi|^2_0=&\big|\int_\Omega\bar{f}\cdot\alpha e^{-\varphi}dx\big|^2_0\\
=&~\big|\int_\Omega\bar{f}\cdot
\frac{1}{\sqrt{\Delta\varphi}}\cdot\alpha\cdot\sqrt{\Delta\varphi}\cdot e^{-\varphi}dx\big|^2_0\\
\leq&~\big\|\bar{f}\frac{1}{\sqrt{\Delta\varphi}}\big\|^2\cdot\big\|\alpha\cdot\sqrt{\Delta\varphi}\big\|^2\\
=&~\int_\Omega\big|\frac{\bar{f}}{\sqrt{\Delta\varphi}}\big|^2_0e^{-\varphi}dx\cdot
\int_\Omega\big|\alpha\cdot\sqrt{\Delta\varphi}\big|^2_0e^{-\varphi}dx\\
\leq & c\|\overline{D}^*_\varphi\alpha\|^2.
\end{split}\nonumber
\end{equation}
The proof is completed with Theorem \ref{thm1}.
\end{proof}

It should be noticed that when $n=1$, $I_3=0$. Then it comes from equation (\ref{38}) that the H\"ormander's $L^2$ theorem in $\R^{2}$ could be described which equals the classical H\"ormander's $L^2$ theorem in $\mathbb{C}$.
\begin{cor}\label{thm3}
Given $\varphi\in C^2(\Omega,\mathbb{R})$ with $\Omega$ being an open subset of $\R^{2}$; $\Delta\varphi\geq0$. Then for all $ f\in L^2(\Omega,\A,\varphi)$ with $\int_\Omega\frac{|f|^2_0}{\Delta\varphi}e^{-\varphi}dx=c<\infty$, there exists a $u\in L^2(\Omega,\A,\varphi)$ such that $$\overline{D}u=f$$ with
$$\|u\|^2\leq\int_\Omega\frac{|f|^2_0}{\Delta\varphi}e^{-\varphi}dx.$$
\end{cor}

\section{Conclusion}
 In this paper, based on the H\"ormander's $L^2$ theorem in complex analysis, the H\"ormander's $L^2$ theorem for Dirac operator in $\R^{n+1}$  has been obtained by Clifford algebra. When $n=1$, the result is equivalent to the classical H\"ormander's $L^2$ theorem in complex variable. Moreover, for any $f$ in $L^2$ space over a bounded domain with value in Clifford algebra, there is a weak solution of Dirac operator with the solution in the $L^2$ space as well. The potential applications of the results will be studied in our future work.

\begin{acknowledgements}
This work was supported by the National Natural Science Foundations of China (No. 11171255, 11101373) and Doctoral Program Foundation of the Ministry of Education of China (No. 20090072110053).
\end{acknowledgements}

%\begin{acknowledgements}
%\end{acknowledgements}

%\bibliography{biblio_journal}   % name your BibTeX data base

\begin{thebibliography}{8}
\providecommand{\natexlab}[1]{#1}
\providecommand{\url}[1]{{#1}}
\providecommand{\urlprefix}{URL }
\expandafter\ifx\csname urlstyle\endcsname\relax
  \providecommand{\doi}[1]{DOI~\discretionary{}{}{}#1}\else
  \providecommand{\doi}{DOI~\discretionary{}{}{}\begingroup
  \urlstyle{rm}\Url}\fi
\providecommand{\eprint}[2][]{\url{#2}}

\bibitem[{Brackx et~al(1982)}]{c}
Brackx F, Delanghe R, Sommen F (1982) Clifford Analysis, Research Notes in
  Mathematics. London, Pitman

  \bibitem[{De Ridder et~al(2012)}]{mz}   De Ridder H,  De Schepper H, Sommen F (2012) Fueter polynomials in discrete Clifford analysis. Mathematische Zeitschrift 272 (2012) :253--268.

\bibitem[{Gong et~al(2009)}]{qt}
Gong Y, Leong IT, Qian T (2009) Two integral operators in Clifford analysis. Journal of Mathematical Analysis and
  Applications 354(2):435--444

\bibitem[{H{\"o}rmander(1965)}]{H}
H{\"o}rmander L (1965) $l^2$ estimates and existence theorems for the operator.
  Acta Mathematica 113(1):89--152

\bibitem[{Huang et~al(2006)}]{c2}
Huang S, Qiao YY, Wen GC (2006) Real and Complex Clifford Analysis, Advances in
  Complex Analysis and Its Applications. New York, Springer

\bibitem[{Qian and Ryan(1996)}]{J6}
Qian T, Ryan J (1996) Conformal transformations and Hardy spaces arising in
  Clifford analysis. Journal of Operator Theory 35(2):349--372

\bibitem[{Ryan(1990)}]{J2}
Ryan J (1990) Iterated Dirac operators in $C^n$. Zeitschrift f{\"u}r Analysis und ihre Anwendungen 9:385--401





\bibitem[{Ryan(1995)}]{J4}
Ryan J (1995) Cauchy-Green type formulae in Clifford analysis. Transactions of
  the American Mathematical Society 347(4):1331--1342

\bibitem[{Ryan(2000)}]{J5}
Ryan J (2000) Basic Clifford analysis. Cubo Matem{\'a}tica Educacional 2:226--256

\end{thebibliography}

%\begin{thebibliography}{}

% etc
%\end{thebibliography}

\end{document}